\documentclass[11pt]{amsart}
\usepackage[usenames,dvipsnames]{color}
\usepackage[table]{xcolor}
\usepackage{lscape}
\usepackage{amsfonts,epsfig}
\usepackage{latexsym}
\usepackage{amssymb}
\usepackage{amsmath}
\usepackage{amsthm}
\usepackage{graphics}
\usepackage[all]{xy}
\usepackage[T2A]{fontenc}
\usepackage{multirow}
\usepackage{array,booktabs,ctable}
\usepackage{enumerate}
\usepackage{booktabs}
\usepackage{longtable}
\usepackage{color}
\usepackage{multirow}
\usepackage{caption}
\usepackage{scrextend}
\usepackage{comment}

\newcommand{\smallmat}[4]{\left(\begin{smallmatrix}#1&#2\\#3&#4\end{smallmatrix}\right)}

\usepackage[pagebackref]{hyperref}
\usepackage{xcolor,colortbl}

\makeatletter
\newcommand\footnoteref[1]{\protected@xdef\@thefnmark{\ref{#1}}\@footnotemark}
\makeatother

\usepackage[croatian,english]{babel}
\usepackage[utf8]{inputenc}
\usepackage[T1]{fontenc}
\usepackage{amsmath}
\usepackage{soul}

\usepackage{booktabs}
\usepackage{longtable}
\usepackage{multirow}

\usepackage{ifthen}
\usepackage{graphpap}

\newcommand{\Q}{\mathbb{Q}}
\newcommand{\Z}{\mathbb{Z}}

\newcommand{\F}{\mathbb{F}}
\newcommand{\Qbar}{{\overline{\mathbb Q}}}
\newcommand{\Kbar}{{\overline{K}}}
\DeclareMathOperator{\GL}{GL}
\DeclareMathOperator{\SL}{SL}
\DeclareMathOperator{\tors}{tors}

\newtheorem{tm}{Theorem}[section]
\newtheorem{proposition}[tm]{Proposition}
\newtheorem{lemma}[tm]{Lemma}

\newtheorem{conjecture}[tm]{Conjecture}
\theoremstyle{definition}
\newtheorem{definition}{Definition}
\theoremstyle{remark}

\newtheorem{remark}[tm]{Remark}

\DeclareMathOperator{\Gal}{Gal}

\DeclareMathOperator{\im}{im}
\newcommand{\Aut}{\operatorname{Aut}}

\title[An algorithm for determining torsion growth of elliptic curves]{An algorithm for determining torsion growth of elliptic curves}

\author{Enrique Gonz\'alez--Jim\'enez}
\address{Universidad Aut{\'o}noma de Madrid, Departamento de Matem{\'a}ticas, Madrid, Spain}
\email{enrique.gonzalez.jimenez@uam.es}
\author{Filip Najman}
\address{University of Zagreb, Bijeni\v{c}ka cesta 30, 10000 Zagreb, Croatia}
\email{fnajman@math.hr}
\thanks{The first author was partially supported by the grant PGC2018--095392--B--I00 (MCIU/AEI/FEDER, UE). The second author gratefully acknowledges support from the QuantiXLie Center of Excellence, a project
co-financed by the Croatian Government and European Union through the
European Regional Development Fund - the Competitiveness and Cohesion
Operational Programme (Grant KK.01.1.1.01.0004) and by the Croatian Science Foundation under the project no. IP-2018-01-1313.}
\date{\today}

\begin{document}

\begin{abstract}
We present a fast algorithm that takes as input an elliptic curve defined over $\Q$ and an integer $d$ and returns all the number fields $K$ of degree $d'$ dividing $d$ such that $E(K)_{\tors}$ contains $E(F)_{\tors}$ as a proper subgroup, for all $F \varsubsetneq K$. We ran this algorithm on all elliptic curves of conductor less than 400.000 (a total of 2.483.649 curves) and all $d \leq 23$ and collected various interesting data. In particular, we find a degree 6 sporadic point on $X_1(4,12)$, which is so far the lowest known degree a sporadic point on $X_1(m,n)$, for $m\geq 2$.
\end{abstract}

\maketitle

\section{Introduction}

Let $E$ be an elliptic curve defined over a number field $K$. The Mordell--Weil Theorem states that the set $E(K)$ of $K$-rational points is a finitely generated abelian group. Denote by $E(K)_{\tors}$ the torsion subgroup of $E(K)$. One of the main goals in the theory of elliptic curves is to determine $E(K)_{\tors}$, or in more generality, all possible torsion groups of all elliptic curves over all number fields of a given degree.

Let $d$ a positive integer and $\Phi(d)$ be the set of groups, up to isomorphism, that occur as torsion groups of some elliptic curve defined over a number field of degree $d$. Note that the set $\Phi(d)$ is finite thanks to Merel's uniform boundedness theorem \cite{merel}. These sets have so far been determined for only\footnote{M. Derickx, A. Etropolski, M. van Hoeij, J. Morrow and D. Zureick-Brown have announced results for $d=3$.} $d\leq 2$  \cite{maz,kam,km}. For degree $d=1,2$, each group in $\Phi(d)$ occurs for infinitely many $\Qbar$-isomorphism classes of elliptic curves, but for $d=3$ this is not the case (see \cite[Theorem 1]{naj} and \cite[Theorem 3.4]{jks}). Therefore we define $\Phi^\infty(d)\subseteq \Phi(d)$ to be the set of groups that arise for infinitely many $\Qbar$-isomorphism classes of elliptic curves. While $\Phi(d)$ is not completely known even for $d=3$, $\Phi^\infty(d)$ is known for $d \le 6$ \cite{jks,jkp-quartic,DS16}.

A slightly different approach is to consider only elliptic curves over $\Q$ under base change to number fields of a given degree. Let $d$ be a positive integer and $\Phi_\Q(d)\subseteq \Phi(d)$ be the set of groups, up to isomorphism, that occur as the torsion group $E(K)_{\tors}$ of an elliptic curve $E$ defined over $\Q$ base changed to a number field $K$ of degree $d$. Notice that $\Phi_\Q(d)$ does not have to be contained in $\Phi^\infty(d)$, as the group $\Z/21\Z$ shows\footnote{The second author showed in \cite{naj} that the elliptic curve with LMFDB label \href{http://www.lmfdb.org/EllipticCurve/Q/162b1}{\texttt{162.c3}} has torsion subgroup $\Z/21\Z$ defined over the cubic field $\Q(\zeta_{9})^+ = \Q(\zeta_9+\zeta_9^{-1})$ where $\zeta_9$ is a primitive $9$-th root of unity.} for $d=3$, and $\Phi^\infty(d)$ does not have to be contained in $\Phi_\Q(d)$ as the group $\Z/15\Z$ shows for $d=2$ (see \cite[Theorem 1]{naj} and \cite{km}).

Similarly, for a fixed $G \in \Phi(1)$, let $\Phi_\Q(d,G)$ be the subset of $\Phi_\Q(d)$ consisting of all possible torsion groups $E(K)_{\tors}$ of an elliptic curve $E$ defined over $\Q$ such that $E(\Q)_{\tors}= G$ base changed to $K$, a number field of degree $d$. The sets $\Phi_\Q(d)$ and $\Phi_\Q(d,G)$, for any $G \in \Phi(1)$, have been completely determined for $d=2,3,4,5,7$ in a series of papers \cite{naj, GJT14,GJT15,GJNT15,chou, GN?,GJ17}. Moreover, in \cite{GN?} it has been established  that $\Phi_\Q(d)=\Phi(1)$ for any positive integer $d$ whose prime divisors are greater than $7$.

Let $E$ be an elliptic curve defined over $\Q$ and let $K$ a number field. We say that there is {\it torsion growth over $K$} if $E(\Q)_{\tors}\subsetneq E(K)_{\tors}$. One can easily work out that there is torsion growth (of the $2$-primary torsion) in at least one number field of degree $2$, $3$, or $4$. On the other hand, there is no torsion growth in number fields of degree only divisible by primes $>7$ (cf. \cite[Theorem 7.2(i)]{GN?}).

The purpose of this paper is to develop a fast algorithm, usable in practice, which for a given elliptic curve $E$ defined over $\Q$ and a positive integer $d$ finds all the pairs $(K,H)$ where $K$ is a number field of degree dividing $d$ and $E(K)_{\tors}\simeq H \supsetneq E(\Q)_{\tors}$. Of course, the set of such number fields can be infinite if there exists a number field $F$ of degree $d'$, where $d'$ divides $d$ and $d'<d$ such that $E(F)_{\tors}\supsetneq E(\Q)_{\tors}$; then every number field $K\supseteq F$ of degree $d$ will have the desired property. To circumvent this problem, we will say that $E$ has \textit{primitive torsion growth} over a number field $K$ if $E(F)_{\tors} \subsetneq E(K)_{\tors}$, for all subfields $F\subsetneq K$. For a prime $\ell$ we say that $E$ has \textit{primitive $\ell$-power torsion growth} if $E(F)[\ell^\infty] \subsetneq E(K)[\ell^\infty]$, for all subfields $F\subsetneq K$.

It is an easy corollary of Merel's theorem \cite{merel} that for a given integer $d$ the list of number fields where primitive torsion growth occurs will be finite. The existence of such an algorithm is obvious: for every integer $d$, by the aforementioned theorem of Merel, there exists an effective bound $B_d$ such that $\#E(K)_{\tors}\leq B_d$. So to determine the number fields $F$ where torsion growth occurs one does the following:\\
\indent $\bullet$ For all prime powers $\ell^n\leq B_d$ do:\\
\indent \indent $\bullet$ factor the $\ell^n$-th division polynomial $\psi_{\ell^n}$ and check whether there are any irreducible factors of degree $d'$ dividing $d$.\\
\indent \indent  $\bullet$ If no, move on to the next prime power. If yes, for all irreducible factors $f$ of degree $d' \,|\, d$ do:
\begin{itemize}
\item[-]  Construct the number field $F$ whose minimal polynomial is $f$ - this will be the field of definition of the $x$-coordinate of a $\ell^n$-torsion point $P$ of $E$.
\item[-]  Check whether $P$ is defined over $F$, if yes add $F$ to the set that will be the output. If $P$ is not defined over $F$, then check whether $2d'$ divides $d$, if yes, then add $\Q(P)$ (which will be obtained from $F$ by adjoining the $y$-coordinate of $P$ to $F$) to the output set.

\item[-] If a point of order $\ell^n$ was constructed in the previous step, check whether the full $\ell^n$-torsion of $E$ is defined over a number field of degree dividing $d$, by checking whether the degree of the splitting field of $\psi_{\ell^n}$ or an appropriate degree 2 extension divides $d$.
\end{itemize}

However, if implemented as stated above, this algorithm would not be very useful in practice. The main obstacle would be factoring division polynomials, as $\psi_n$ is a polynomial of degree $\frac{n^2-1}{2}$ for $n$ odd, and the values $n$ that need to be checked will grow exponentially in $d$.

Our algorithm will use information that can be obtained from the images of mod $n$ Galois representations attached to $E$ to avoid factoring division polynomials wherever possible. To make the algorithm usable in practice we will add a number of if-then conditions that will rule out most of the integers $n$ that need to be checked using results from \cite{GN?} and results that we develop for this purpose in Section \ref{sec:if-then}. 

One of the main motivations of this paper is to run the algorithm on all elliptic curves of conductor less than 400.000 (see \cite{cremonaweb,lmfdb}) and for each curve within determine all the number fields of degree $\le 23$ over which there is primitive torsion growth. In Section \ref{sec:lmfdb}  we present the most interesting data coming out of these computations.  The main results appear in Table \ref{table}. We obtain sets contained in $\Phi_\Q(d)$ for $d\le 23$ and our data motivates us to conjecture that we have in fact obtained all of $\Phi_\Q(d)$ for $d\le 23$ (see Conjecture \ref{main_conjecture}). We can also see that there is much more torsion growth and it is much more complex when $d$ is divisible by powers of $3$ and especially $2$. Moreover we find two elliptic curves defined over $\Q$ with torsion $\Z/4\Z\times \Z/12\Z$ over a degree 6 number field and prove that these are the only two such curves. By \cite{DS16}, there are only finitely many elliptic curves over sextic fields (without supposing that they are defined over $\Q$) with this torsion group, so these curves give us examples of sporadic points of degree 6 on $X_1(4,12)$. This is the lowest known degree of a sporadic point on a modular curve $X_1(m,n)$, for $m|n$ and $m\geq 2$. 

\subsubsection*{Notation} Specific elliptic curves mentioned in this paper will be referred to by their LMFDB label and a link to the corresponding \href{http:/www.lmfdb.org/EllipticCurve/Q/}{LMFDB} page \cite{lmfdb} will be included for the ease of the reader. Conjugacy classes of subgroups of $\GL_2(\Z/\ell \Z)$ will be referred to by the labels introduced by Sutherland in \cite[\S 6.4]{Sutherland2}. We write $G=H$ (or $G\leq H$) for the fact that $G$ is isomorphic to $H$ (or to a subgroup of $H$ resp.) without further detail on the precise isomorphism. 

\section{Auxiliary results}
In this section, we prove a series of results that will make it possible to replace costly factorizations of division polynomials by simple if-then checks. This will be useful in the computations described in Section \ref{sec:lmfdb}.

Let $E$ be an elliptic curve defined over a number field $K$, $n$ a positive integer and $\Kbar$ a fixed algebraic closure of $K$. The absolute Galois group $G_K:=\Gal(\Kbar/K)$ acts on $E[n]$, inducing a \textit{mod $n$ Galois representation attached to $E$}
$$
\overline{\rho}_{E,n}\,:\, G_K\longrightarrow \Aut(E[n]).
$$
Fixing a basis $\{P,Q\}$ of $E[n]$, we identify $ \Aut(E[n])$ with $\GL_2(\Z/n\Z)$. Therefore we can view $\overline{\rho}_{E,n}(G_K)$ as a subgroup of $\GL_2(\Z/n\Z)$, determined uniquely up to conjugacy in $\GL_2(\Z /n\Z)$, and denoted by $G_E(n)$ from now on.

For elliptic curves over $\Q$, we conjecturally (see \cite[Conjecture 1.1]{Sutherland2} and \cite[Conjecture 1.12.]{zyw}) know all the mod $\ell$ Galois representations attached to non-CM elliptic curves over $\Q$.

\begin{conjecture}
\label{SUC}
Let $E/\Q$ be a non-CM elliptic curve, $\ell\geq 17$ a prime and $(\ell, j_E)$ not in the set
$$
\left\{(17,-17\cdot 373^3/2^{17}), (17,-17^2\cdot 101^3/2), (37,-7\cdot 11^3), (37, -7\cdot 137^3\cdot2083^3)  \right \},
$$
then $G_E(\ell)=\GL_2(\F_\ell)$.
\label{serre_conj}
\end{conjecture}

For a prime $\ell$, $\rho_{E,\ell}:G_K \rightarrow \GL_2(\Z_\ell)$ will denote the $\ell$-adic representation attached to $E$ (again we assume that we have fixed a basis for the Tate module $T_\ell(E)$). We say that the $\ell$-adic representation of $E$ is \textit{defined modulo $\ell^n$} if for all $m\geq n$ we have $G_E(\ell^{m+1})\geq I_2+\ell^mM_2(\Z/\ell^{m+1}\Z)$, where $I_2$ is the identity matrix.

\begin{proposition}
\label{prop:divisibility}
Let $E$ be an elliptic curve defined over a number field $K$ such that its $\ell$-adic representation is defined modulo $\ell^n$. Then for any point $P\in E(\Kbar)$ of order $\ell^{n+1}$, we have $[K(P):K(\ell P)]=\ell^2$.
\end{proposition}
\begin{proof}
We need to prove that $I_2+\ell M_2(\Z/\ell^{n+1}\Z)$ acts transitively on the solutions of $\ell X=P$ (where the action of $I_2+\ell^nM_2(\Z/\ell^{n+1}\Z)$ on the $\Z/\ell^n\Z$-module of the solutions of $\ell X=P$ is defined in the obvious way). The $G_K$-module $E[\ell^{n+1}]$ is isomorphic to $(\Z/\ell^{n+1}\Z)^2$, and we choose an isomorphism sending $P$ to $(\ell,0)$ and study the action of $I_2+\ell M_2(\Z/\ell^{n+1}\Z)$ on the $\ell^2$ solutions of the equation $\ell X=(\ell,0)$. One easily sees that already the subgroup of $I_2+\ell^n M_2(\Z/\ell^{n+1}\Z)$ generated by $\smallmat{1}{\ell^n}{0}{1}$ and $\smallmat{\ell^n+1}{0}{0}{1}$
acts transitively on the solutions of the equation $\ell X=(\ell,0)$.
\end{proof}

For easier reference we state and prove the following lemma which will follow from standard group-theoretic arguments.

\begin{lemma}\label{lem:index}
Let $E$ be an elliptic curve without $CM$ defined over a number field $K$ and $\ell\geq 5$ a prime such that $v_\ell([\GL_2(\Z_\ell):\rho_{E,\ell}(G_\Q)])=n$. 
  Then the $\ell$-adic Galois representation of $E/K$ is defined modulo $\ell^m$ for some $m\leq n+1$.
\end{lemma}
\begin{proof}
Define $V_k:=I_2+\ell^kM_2(\Z/\ell^{k+1}\Z)$ and $G:=\rho_{E,\ell}(G_\Q)$. 
Let $\rho_k :\GL_2(\Z/\ell^{k+1}\Z)\rightarrow \GL_2(\Z/\ell^{k}\Z)$
be the reduction mod $\ell^k$ map. Then $\ker \rho_k=V_k$.

We use the fact, as explained in the proof of \cite[Lemma 3, IV-23]{serre3}, that if $V_{m} \subset G_E(\ell^{m+1})$, then $V_{k} \subset G_E(\ell^{k+1})$ for all $k\geq m$. It follows that if $G$ is defined modulo $\ell^m$, then we have $\ker (\rho_k|_{G_E(\ell^{k+1})})=V_k$ for all $k\geq m$. So if $\rho_m^{-1}(G_E(\ell^m))=G_E(\ell^{m+1})$, then $\rho_k^{-1}(G_E(\ell^k))=G_E(\ell^{k+1})$ for all $k\geq m$.


This implies that if $G$ is defined modulo $\ell^m$, then $\rho_k^{-1}(G_E(\ell^k))\neq G_E(\ell^{k+1})$ (and hence $\rho_k^{-1}(G_E(\ell^k))$ is of index $\ell^i$ for some $1\leq i \leq 4$ in $G_E(\ell^{k+1})$), for all $1\leq k\leq m-1$.

Suppose now that $G$ is not defined modulo $\ell^{m}$ for any $1 \leq m \leq n+1$. This implies that $$[\GL_2(\Z/ \ell^{k+1}\Z):G_E(\ell^{k+1})]\geq \ell[\GL_2(\Z/ \ell^{k}\Z):G_E(\ell^{k})]$$
for all $1\leq k \leq n+1$. This implies that $v_\ell\left([\GL_2(\Z_\ell):\rho_{E,\ell}(G_\Q)]\right)\geq n+1$, which is a contradiction.





\end{proof}

\begin{lemma}
Let $\ell \geq 3$ be a prime and $E/\Q$ an elliptic curve. Then if $G_E(\ell)=\GL_2(\F_\ell)$ and $P\in E(\overline \Q)$ is a point of order $\ell^2$, then $[\Q(P):\Q]=\ell^2(\ell^2-1)$.
\end{lemma}

\begin{proof}
If $\ell \geq 5$, then it follows from \cite[Lemma 3, IV-24]{serre3} that if $G_E(\ell)=\GL_2(\F_\ell)$, then $\rho_{E,\ell}$ is surjective. It follows that the $\ell$-adic representation is defined modulo $\ell$, so the lemma follows from Proposition \ref{prop:divisibility}. For $\ell=3$, if $G_E(9)=\GL_2(\Z/9\Z)$, then the conclusion is the same as before, while if $G_E(9)\neq \GL_2(\Z/9\Z)$ then it follows from \cite{elkies} that $G_E(9)=G$, where $G$ is a (unique up to conjugacy) subgroup $G$ of $\GL_2(\Z/9\Z)$ generated by
$\smallmat{4}{5}{4}{4}$ and $\smallmat{4}{5}{8}{6}$.
One easily checks that this group acts transitively on the 72 points of order 9 in $E(\overline \Q)$, so the  $[\Q(P):\Q]=72$ for all points of order 9 (using the same argumentation as in \cite[Section 5]{GN?}).

\end{proof}

\begin{lemma}
\label{definition_of_pn}
Let $\ell$ be a prime, $E$ an elliptic curve defined over a number field $K$, $P\in E(\Kbar)$ a point of order $\ell^k$, $k\geq 2$ and suppose   $\ell^{j}P\in E(K)$ for some $j<k$ and such that $\ell^{k-j}>2$. Then $K(x(P))=K(P)$.
 \end{lemma}
\begin{proof}
Obviously $[K(P):K(x(P))]=1$ or $2$. Suppose $[K(P):K(x(P))]=2$ and let $1\neq \sigma \in \Gal(K(P)/K(x(P)))$. Then we have $\sigma(x(P))=x(P)$, so $\sigma(y(P))=y(-P)$ and hence $\sigma(P)=-P$, as $\sigma\neq 1$. But we have
$$-\ell^{j}P=\ell^{j}(\sigma(P))=\sigma (\ell^{j}P)=\ell^{j}P,$$
where the last equation follows from the fact that $\ell^{j}P\in E(K)$. Since by assumption $\ell^{j}P$ is a point of order $>2$, this is a contradiction.
\end{proof}

The most time-consuming part of our algorithm is determining the existence of points of order $\ell^k$ for $k\geq 2$, and the fields over which such points live if they exist.

We now prove results that will prove the non-existence of points of certain orders $\ell ^k$ over number fields of relatively small degree $d$.

\subsection{Points of order $125$}

\begin{proposition}
\label{lem:5-power}
Let $E/\Q$ be an elliptic curve and $K$ a number field of degree $<50$. Then $E(K)$ does not have a point of order $125$.
\end{proposition}
\begin{proof}
Let $P$ be a point of order $125$. First consider the case when $E$ has a $5$-isogeny over $\Q$. Let $d$ be the power of 5 in $[\Aut_{\Z_5}T_5(E):\im \rho_{E,5}]$ (note that this index is finite as elliptic curves with CM do not have $5$-isogenies over $\Q$). By \cite[Theorem 2]{Greenberg}, $d$ is at most $5$, and we conclude by Lemma \ref{lem:index} that the $\ell$-adic representation of $E$ is defined modulo $25$. From here it follows by Proposition \ref{prop:divisibility} that $[\Q(P):\Q(5 P)]=25$. Since there exist no points of order $25$ on elliptic curves over quadratic fields \cite{kam,km}, we have $[\Q(P):\Q]>50$.

Suppose now that there is no isogeny of degree 5 over $\Q$. Applying \cite[Theorem 2.1]{loz2} (with $L=\Q$, $p=5$, $a=1$ and $n=3$), we obtain that $[\Q(P):\Q]$ is divisible by $25$. From \cite[Table 1]{GN?} we see that the field over which an elliptic curve without an isogeny gains a $5$-torsion point is divisible by 2. So we conclude that $[\Q(P):\Q]$ is divisible by 50 and hence $[\Q(P):\Q]\geq 50$.
\end{proof}

\label{sec:if-then}

\subsection{Points of order $49$}


\begin{lemma}
\label{lem:49}
There are no points of order $49$ on an elliptic curve $E/\Q$ over any number field of degree $d<42$.
\end{lemma}
\begin{proof}
Let us split the proof in two cases depending if $E$ has a 7-isogeny or not. If $E$ has a 7-isogeny, then by the results of \cite{GRSS} the $7$-adic representation is either as large as possible or the curve has $j_E=-15^3$ or $255^3$. If the representation is as large as possible, then by Proposition \ref{prop:divisibility} we have $[\Q(P):\Q(7P)]=49$, eliminating this case. If $j_E=-15^3$ or $255^3$, we explicitly \href{http://matematicas.uam.es/~enrique.gonzalez.jimenez/research/tables/algorithm/Lemma_2_8_order49.txt}{\color{blue}check} that $[\Q(P):\Q]\geq 147$.

Finally, suppose that $E$ does not have a $7$-isogeny and let $P$ be a point of order $49$ of $E$. 
By \cite[Theorem 2.1]{loz2}, we get that $[\Q(P):\Q(7P)]$ is divisible by $7$. So if $[\Q(P):\Q]=[\Q(P):\Q(7P)][\Q(7P):\Q]<42$, then it would follow that $[\Q(7P):\Q]<6$. By looking at \cite[Table 1]{GN?} we see that this is only possible when $G_E(7)$ is a Borel subgroup, which is a contradiction, since then $E$ would have a $7$-isogeny over $\Q$.

\end{proof}

\subsection{Points of order $\ell^2$ for $\ell > 7$}

\begin{lemma}
\label{lem:p-power}
There are no points of order $\ell^2$ for $\ell\geq 11$ on an elliptic curve $E/\Q$ over any number field of degree $d< 55$.
\end{lemma}
\begin{proof}
We divide the proof into two cases: when $E$ has CM and when it doesn't.

Suppose first that $E$ doesn't have CM. Let $P$ be a point of order $\ell^2$. If $E$ has a $\ell$-isogeny over $\Q$ and does not have CM, by the results of \cite{Greenberg}, it follows that the $\ell$-adic image is defined mod $\ell$, from which it follows by Proposition \ref{prop:divisibility} that $[\Q(P):\Q(\ell P)]$ is divisible by $\ell^2$. On the other hand, if there are no $\ell$-isogenies over $\Q$, then we have that $[\Q(\ell P):\Q]\geq 55$ for $\ell=11$ by \cite[Table 1]{GN?}, $[\Q(\ell P):\Q]\geq 72$ for $\ell=13$ by \cite[Table 2]{GN?} and $[\Q(\ell P):\Q]\geq (\ell^2-1)/3$ for $\ell> 13$ by \cite[Theorems 3.2 and 5.6]{GN?}.

Suppose now that $E$ has CM by an order $\mathcal{O}$. Let $F=\Q(P)$, where $P$ is of order $\ell^k$, $k\geq 2$, and let $K=\mathcal O \otimes_\Z \Q$ be the CM field of $E$. If $j(E) \neq 0,1728$ or $\ell \neq 13$ it follows from \cite[Theorem 6.2]{BC} that $FK$ is of degree $\geq 110$ and hence $[F:\Q]\geq 55$.

Suppose from now on that $\ell=13$. If $K\nsubseteq F$ and $j(E)=0$ or $1728$, we have $(\ell,\Delta(\mathcal{O}))=1$ so we can apply \cite[Therorem 4.8 c)]{bcs} to show that $\Q(\zeta_{\ell^k})^+$ is strictly contained in $F$, from which it follows that $[F:\Q]\geq 312$.

Suppose $K\subset F$. If $j(E)=1728$, then \cite[Theorem 6.2]{BC} gives us that $[F:\Q]\geq 78$. Finally, suppose $j(E)=0$. From \cite[Theorem 6.2]{BC} it follows that for any elliptic curve with $j(E)=0$ with a point of order $169$ over a number field $F$ containing $K$, we have that $52\mid [F:\Q]$. Suppose $E/F$ is such an elliptic curve; i.e. $j(E)=0$, $E(F)_{\tors}$ has a point of order 169 and $[F:K]=26$. We claim that $E$ cannot be a base change of an elliptic curve defined over $\Q$.
Since $F$ is a subfield of $\Q(E[13])$, by the theory of complex multiplication it is Abelian over $K$ and so $\Gal(F/K)\simeq \Z/26\Z$. It follows by Galois theory that there exists a field $K\subset F'\subset F$ where $[F:F']=2$. We can write $F=F'(\sqrt \delta)$ for some $\delta \in F'$. Let $E^\delta$ be the quadratic twist of $E$ by $\delta$. If $E$ was defined over $\Q$, then we would have (see for example \cite[Lemma 1.1]{lol})
$$E(F)[169]\simeq E(F')[169] \oplus E^\delta(F')[169],$$
which now implies that there exists an elliptic curve with $j(E)=0$ and a point of order $169$ over $F'$, contradicting \cite[Theorem 6.2]{BC}.

\end{proof}

\subsection{Points of order $37$}
The following lemma allows us to deal with points of order $37$ over number fields of degree $12$, which is the smallest degree over which an elliptic curve defined over $\Q$ can have a point of order $37$.
\begin{lemma} \label{lem:37}
Let $E/\Q$ be an elliptic curve. Then $E$ has a point of order $37$ over a degree $12$ number field $K$ if and only if $j_E=-7\cdot 11^3$. Moreover, $K$ has to be $K=\Q(\alpha,\sqrt{\frac 1 d \cdot f(\alpha)})$ where $f(x)=x^3 - 1155x + 16450$, $d\in \Q$ is such that $E$ is $\Q$-isomorphic to the elliptic curve $dy^2=f(x)$ and $\alpha$ is a root of the irreducible polynomial
$$
g(x)=x^6 - 210x^5 - 8085x^4 + 125300x^3 + 4251975x^2 - 16133250x - 408849875.
$$
In particular, $E(K)_{\tors}=\Z/37\Z$.
\end{lemma}
\begin{proof}
From \cite[Table 2]{GN?} it follows that $E$ has a point of order 37 over a degree 12 field if and only if $G_E(37)=\texttt{37.B.8.1}$, which happens if and only if $j_E=-7\cdot 11^3$ (see \cite[Theorem 1.10. (ii)]{zyw}). We note that the elliptic curve $E':y^2=f(x)$ has $j_{E'}=-7\cdot 11^3$ and therefore there exists a number field $L$ of degree $12$ such that $E'$ has a point of order $37$ over $L$ (see \cite[Section 6]{naj}). We  \href{http://matematicas.uam.es/~enrique.gonzalez.jimenez/research/tables/algorithm/Lemma_2.10_11.txt}{\color{blue}have} that $g(x)$ is an irreducible factor of the $37$-division polynomial of $E'$. In particular $\alpha=x(P)$ where $P$ is a point of order $37$ in $E'$ and $L=\Q(P)=\Q(\alpha,\sqrt{f(\alpha)})$. Now if $E/\Q$ is an elliptic curve with $j_E=-7\cdot 11^3$, it will be a quadratic twist of $E'$; thus $E$ will have a model $E\,:\,dy^2=f(x)$ for some $d\in\Q$. In particular, $R=(\alpha,\sqrt{\frac 1 d\cdot f(\alpha)})$ is a point of order $37$ on $E$. Then we obtain $K=\Q(R)$ and get the desired result.

Let us prove $E(K)_{\tors}=\Z/37\Z$. The curve $E$ cannot have full $37$-torsion over $K$ by the Weil pairing and cannot have a point of order $37^2$ by Lemma \ref{lem:p-power}. The set of non-surjective primes only depends on the $j$-invariant of $E$ (\cite[Lemma 5.27]{Sutherland2}). Therefore it is enough to compute this set for a single elliptic curve with that $j_E=-7\cdot 11^3$. We have that the elliptic curve $E'$ of minimal conductor with $j_{E'}=-7\cdot 11^3$ has LMFDB label \href{http://www.lmfdb.org/EllipticCurve/Q/1225h1}{\texttt{1225.b2}}.

We see in the LMFDB\footnote{Note that the data for non-CM elliptic curves over $\Q$ in the LMFDB provably includes
all p for which the mod-p representation is non-surjective (this has
been verified using Zywina's algorithm \cite{zyw}, see \url{https://www.lmfdb.org/EllipticCurve/Q/Reliability}).} (or alternatively explicitly compute) that $37$ is the only non-surjective prime for this elliptic curve. So if $E(K)$ had a point $P$ of order $\ell\neq 37$, $\Q(P)$ would have to be a subfield of $K$ and $\ell^2-1$ would have to divide $12$. We see that the only possibility is that $\ell=2$. But the field $\Q(P)$ generated by a point of order $2$ will not be Galois over $\Q$, since the mod $2$ representation is surjective, and hence cannot be a subfield of the cyclic field $K$ (we see that $K$ is cyclic as it is generated by a point lying in the kernel of an isogeny, see \cite[Lemma 4.8]{DLNS}).

\end{proof}

\subsection{Points of order $17$}
We obtain similar results as in Lemma \ref{lem:37}, but for order $17$ and for number fields of degree $8$, which is the smallest degree over which an elliptic curve defined over $\Q$ can have a point of order $17$.
\begin{lemma} \label{lem:17}
Let $E/\Q$ be an elliptic curve. Then $E$ has a point of order $17$ over a degree $8$ number field $K$ if and only if $j_E=-17\cdot 373^3/2^{17}$. Moreover, $K$ has to be $K=\Q(\alpha,\sqrt{d\cdot f(\alpha)})$ where $f(x)=x^3 - 95115 x - 12657350$, $d\in \Q$ is such that $E$ is $\Q$-isomorphic to the elliptic curve $dy^2=f(x)$ and $\alpha$ is a root the irreducible polynomial
$$
g(x)=x^4 + 340x^3 + 510x^2 - 5560700x - 237673175.
$$
In particular $E(K)_{\tors}=\Z/17\Z$.
\end{lemma}
\begin{proof} By the same arguments as in Lemma \ref{lem:37}, we get that an elliptic curve $E/\Q$ such that $E$ gains a point of order $17$ over a number field $K$ of degree $8$ has $j_E=-1\cdot 2^{-17}\cdot 17\cdot 373^3$ (see \cite[Table 2]{GN?} and \cite[Theorem 1.10. (i)]{zyw}) and $17$ is the only surjective prime\footnote{This can be read off from LMFDB - see the footnote in Lemma \ref{lem:37}.} for all such curves. Note that in this case the quadratic twist with minimal conductor of $E'$ has LMFDB label  \href{http://www.lmfdb.org/EllipticCurve/Q/14450n1}{\texttt{14450.o2}}.

Let us prove $E(K)_{\tors}=\Z/17\Z$. The curve $E$ cannot have full $17$-torsion over $K$ by the Weil pairing and cannot have a point of order $17^2$ by Lemma \ref{lem:p-power}. So if $E(K)$ had a point $P$ of order $\ell\neq 17$, $\Q(P)$ would have to be a subfield of $K$ and $\ell^2-1$ would have to divide $8$. We see that the only possibility is that $\ell=3$. But there cannot be any points of order $3$ over $K$, as $K$ is cyclic (as it is generated by a point lying in the kernel of an isogeny, see \cite[Lemma 4.8]{DLNS}) and $\Q(P)$ will not be Galois over $\Q$ for any $P\in E[3]$.

\end{proof}

\subsection{Some special degrees}
From the results proved in this section, we immediately obtain the following result.
\begin{lemma} Let $d=22$ or $26$ and $E/\Q$ an elliptic curve. Then there is no primitive torsion growth over any number field of degree $d$. \label{lem:22-26}
\end{lemma}
\begin{proof}
Suppose the opposite, in particular that for some $P\in E(\overline \Q)_{tors}$, we have $[\Q(P):\Q]=d$. Let $K=\Q(P)$. From \cite[Theorem 5.8]{GN?} we see that there is no primitive $\ell$-torsion growth over $K$ for any prime $\ell$. Moreover, we see that there can be no points of order $\ell\geq 11$ over $K$ at all. It remains to check whether the $\ell$-power torsion cannot grow from a subfield of $K$ to $K$ for $\ell \leq 7$. If $P$ is a point of order $\ell^k$, then it would follow that $\Q(P)\subset \Q(E[\ell^k])$ for some $k$. So in particular $[\Q(P):\Q]=d$ divides $[\Q(E[\ell^k]):\Q ]=|G_E(\ell^k)|$. Since $G_E(\ell^k)$ is a subgroup of $\GL_2(\Z/\ell^k \Z)$, it follows that $d$ divides $\# \GL_2(\Z/\ell^k\Z)$. This is easily seen to be a contradiction for all $\ell \leq 7$.
By the same argument, the extension over which a subgroup of the form $\Z/l^m\Z \times \Z/l^n\Z$ is first defined cannot be of degree $22$ or $26$. More generally, one can deduce the same result for a group of the form $\Z/m\Z \times \Z/n\Z$ for integers $n$ and $m$ divisible by multiple primes.
\end{proof}



\section{The algorithm}
In this section we describe our algorithm. We always strive to make the algorithm useful in practice, and not to obtain an algorithm with small worst-case complexity. The reason for this is that in most cases, standard conjectures tell us that certain things will not happen, so we do not worry too much about the run-times of events that are conjecturally impossible. To give an explicit example, it is widely believed (see Conjecture \ref{serre_conj}) that $G_E(\ell)=\GL_2(\F_\ell)$ for all $\ell>37$ and all non-CM elliptic curves over $\Q$. Hence, we focus on trying to quickly prove that indeed $G_E(\ell)=\GL_2(\F_\ell)$, and not worry too much on the run-time of what happens if $G_E(\ell)\neq \GL_2(\F_\ell)$ for $\ell>37$, which, as already noted, conjecturally never happens.

We will use the following notation/definition in the algorithm.

\begin{definition}
For an elliptic curve $E/\Q$ and a positive integer $d$, we define $R(d,E)$ to be the set of primes such that there exists a number field $K$ of degree $d'|d$ such that there is primitive $\ell$-power torsion growth over $K$.
\end{definition}

Recall that in \cite{GN?} the set $R_\Q(d)$ is defined to be the set of all primes $\ell$ such that there exists a point of order $\ell$ on some elliptic curve $E/\Q$ over some number field of degree $d$. Note that $R_\Q(d)$ is unconditionally known for all $d<3.343.296$ (and in the larger cases we know a set containing $R_\Q(d)$), so for all values of $d$ in which one hopes to be able to run the algorithm.

%
%

\medskip

The algorithm consists of 3 sub-algorithms.

\medskip


{\bf \textsc{Algorithm 1: $R(d,E)$}}

{\sc Input:} An elliptic curve $E/\Q$ and integer $d$.

{\sc Output:} The set $R(d,E)$

\begin{enumerate}

\item Set $R(d,E):=\emptyset$.

\item \label{step1} If the largest prime divisor of $d$ is larger than $7$, exit this algorithm  and return $R(d,E)=\emptyset$.

\item Compute $R_\Q(d)$ using \cite[Corollary 6.1]{GN?}.

\item \label{step3} For $\ell \in R_\Q(d)$ compute $G_E(\ell)$.


\item For $\ell\in R_\Q(d)$ compute the degrees $n$ of number fields over which there is $\ell$-torsion, depending on $G_E(\ell)$ using \cite[Table 1 $\&$ 2]{GN?} and \cite[Theorem 3.2]{GN?} for non-CM curves and
    \cite[Theorem 3.6 and 5.6 ]{GN?} for CM curves. If any such $n$ divides $d$, add $\ell$ to $R(d,E)$.
\item[(6)] Return $R(d,E)$.
\end{enumerate}
\begin{remark}
Algorithm 1 is used to determine the (finite) set of primes $\ell$ such that there will be primitive $\ell$-power torsion growth over number fields of degree $d'$ dividing $d$.
\end{remark}

\begin{remark}
Step \eqref{step1} follows from \cite[Theorem 7.1. (i)]{GN?}. In step \eqref{step3}, we compute $G_E(\ell)$ using the algorithm sketched in \cite[1.8.]{zyw}.
\end{remark}

\medskip

{\bf \textsc{Algorithm 2: $\ell$-primary torsion growth}}

In this algorithm we will store a point or points generating the torsion group of $E(K)$. These are necessary for computing the $\ell$-power torsion, but will not be returned in the output of the algorithm (although they could be), as they will not be necessary. We will also store an auxiliary sequence $F$ of pairs $(F_i,(P_i,Q_i))$, where $F_i=\Q(E[\ell^i])$ and $P_i$ and $Q_i$ generate the $\ell^i$-torsion of $E$ and such that $[F_i:\Q]$ divides $d$. In Algorithm 2, $F_i$ will always denote $\Q(E[\ell^i])$.

\medskip

\textsc{Input:} An elliptic curve $E/\Q$, $d\in \Z_+$, a prime $\ell$

\textsc{Output:} A set $A$ of all pairs $(K, T)$ such that $E$ has primitive $\ell$-power torsion growth over $K$, the group $T:=E(K)[\ell^\infty]$ and such that $[K:\Q]$ divides $d$.

\begin{enumerate}
  \item $A:=\emptyset$ and $F:=\emptyset$.
  \item \label{step2.1} If $E(\Q)[\ell]\neq \{0\}$: Set $A:=A \cup (\Q, E(\Q)[\ell],S)$, where $S$ is a set of generators of $E(\Q)[\ell]$. If $\#G_E(\ell)$ divides $d$, then factor $\psi_\ell$, set $F_1=\Q(E[\ell])$\footnote{We have $\Q(E[\ell])=F_1$ by \cite[Lemma 5.17]{Sutherland2}.} to be the field defined by an irreducible factor of degree $>1$ and set $A:=A \cup (F_1, (\Z/\ell\Z)^2, S)$ and $F:=(F_1, S)$, where $S$ is a set of generators of $E[\ell]$.


  \item \label{step2.2}If $E(\Q)[\ell]=\{0\}$: 
      Explicitly determine the triples $(K_i:=\Q(P_i), \Z/\ell\Z, \{P_i\})$ for all $P_i\in E[\ell]$ by factoring the $\ell$-division polynomial $\psi_\ell$, keeping only one number field up to isomorphism. Add all these triples to $A$. For all $K_i$ constructed, check whether $\#G_E(\ell)=[K_i:\Q]$ for any $i$; if yes, change $(K_i, \Z/\ell\Z,\{P_i\})$ to $(K_i, (\Z/\ell\Z)^2, S)$ and $F:=F \cup (F_1, S)$, where $S$ generates $E[\ell]$.


  \item \label{step2.3} Set $k:=2$. Repeat: if ($\ell<11$ or $d\geq 55$) and ($\ell \neq 5$ or $k=2$ or $d\geq 50$) and ($\ell \neq 7$ or $d\geq 42$)
        \begin{itemize}
          \item[(i)] Compute the primitive $\ell^k$-division polynomial $\psi_{\ell^k}/\psi_{\ell^{k-1}}$, as a polynomial in $\Z[x]$, reduce it modulo small primes $p$ of good reduction different from $\ell$, factor it over $\F_p[x]$, and check whether there are any irreducible factors of degree dividing $d$ for each prime $p$. If not, then exit the loop.

          \item[(ii)] \label{stepii} 
              Now for each element $(K_i, T, S)$ that we have in $A$, for each cyclic subgroup of $T$ of order $\ell^{k-1}$ (if it exists): select a generator $Q$. Factor over $K_i$ the polynomial
\begin{equation}
\label{div_pol}
\phi_\ell(x) - x(Q)\psi_\ell(x)^2=g_1(x)\cdot \ldots \cdot g_u(x),
\end{equation}
where $\phi_\ell$ and $\psi_\ell$ are as defined in \cite[Chapter 3.2. p.81]{was} \footnote{We use \cite[Corollary 5.18]{Sutherland2} where possible. By \cite[Theorem 3.6]{was} we have that
              $
x(Q)=\frac{\phi_\ell(x)}{\psi_\ell(x)^2},
$
for any $P=(x,y)$ such that $Q=\ell P$. Using this step is crucial (instead of factoring $\ell^k$-division polynomials) as one uses the polynomial \eqref{div_pol} of degree $\ell^2$ (over number fields) instead of factoring (over $\Q$) the primitive $\ell^k$-division polynomial, which is of degree $\ell^{2k-2}(\ell^2-1)/2$.
}.
Let $P_i$ be a point of order $\ell^k$ such that $x(P_i)$ is a root of $g_i$.

 If $[\Q(P_i):\Q]$ divides $d$, define $T'$ by as follows: if $T$ was $\Z/\ell^{k-1} \Z \times \Z/\ell^j \Z$ for some $j$, then $T':=\Z/\ell^{k} \Z \times \Z/\ell^j \Z$. Add the field $\Q(P_i)$, the subgroup $T'$ and its generators into $A$, where the generators of $T'$ are obtained by taking the generators of $T$ and replacing $\ell P_i$ by $P_i$.


          \item[(iii)] \label{stepiii} For each element $(K_i, \Z/{\ell^k}\Z\times \Z/{\ell^n}\Z, S)$ in $A$, check whether $K_iF_j$ is of degree dividing $d$ for $j=n+1,\ldots, k-1$. If yes, add $(K_iF_j, \Z/{\ell^k}\Z\times \Z/{\ell^j}\Z, S')$ to $A$ and if furthermore $K_iF_j=K_i$, then remove the triple $(K_i, \Z/{\ell^k}\Z\times \Z/{\ell^n}\Z, S)$ from $A$.


          \item[(iv)] Check whether $F_k$ is of degree dividing $d$ by checking whether in $A$ there exists an entry $(K_i, \Z/{\ell^k}\Z\times \Z/{\ell^{k-1}}\Z, S)$; if yes, check whether the element $P\in S$ of order $\ell^{k-1}$ is divisible by $\ell$ over $K_i$. If yes, change the previous entry into $(K_i, \Z/{\ell^k}\Z\times \Z/{\ell^{k}}\Z, S')$, where $S'=\{Q, R\}$ is obtained from $S=\{P,Q\}$, where $Q$ is of order $\ell^k$ and $\ell R=P$ and add $(F_k:=K_i, S')$ to $F$.

          \item[(v)] $k:=k+1$;
        \end{itemize}
        until the first occurrence that there are no points of order $\ell^{k}$ in A.
    \item \label{return_A} Return A.
\end{enumerate}

\begin{remark}
The conditions at the beginning of \eqref{step2.3} come from Lemmas \ref{lem:5-power}, \ref{lem:49} and \ref{lem:p-power},  and make the algorithm much faster for "small" ($<50$) degrees, i.e. in all the ones where it is feasible to use the algorithm in practice.

In \eqref{stepiii} (iii), if $K_iF_j$ is not of degree dividing $d$, then neither is $K_iF_{j+1}$, so we can stop for the smallest $j$ such $K_iF_j$ is not of degree dividing $d$.

In \eqref{return_A}, the generators of the torsion groups can be deleted from $A$, as they will not be used again later.

\end{remark}

\medskip

{\bf \textsc{Algorithm 3: Combining different $\ell$-primary torsion growths}}

\textsc{Input:} A positive integer $d$, a set $A$ of all pairs $(K, T)$ such that $E$ has primitive $\ell$-power torsion growth over $K$ for some prime $\ell$, where $[K:\Q]$ divides $d$, and the group $T:=E(K)[\ell^\infty]$.

\textsc{Output:} A set $B$ of all pairs $(K, T)$ where $E$ has primitive torsion growth over $K$ and such that $[K:\Q]$ divides $d$, and where $T:=E(K)_{\tors}$.

\begin{enumerate}
  \item \label{step3x} To each pair $(K, T)$ previously obtained we adjoin the set $\{(\ell, K^\ell)\}$ where $\ell$ is a prime such that $T$ is an $\ell$-group and $K^\ell:=K$ . So we get triples $(K,T,\{(\ell, K)\})$. For a triple $(K,T,S)$, where $S:=\{(\ell, K)\}$, we will denote by $S':=\bigcup_{a \in S} {a[1]}$ the set of all first coordinates of $S$.
      \item Set $k:=2$; Repeat: new:=false;
  \begin{itemize}
    \item[(i)] For each pair of triples $(K_i,T_i,S_i)$ and $(K_j,T_j,S_j)$ satisfying $|S_i' \cup S_j'|=|S_i\cup S_j|=k$ check whether the degree of $K_iK_j$ divides $d$. If yes put new:=true and construct the triple $(K_iK_j, T, S_i \cup S_j)$ where
        $$T=\prod_{\ell \in (S_i \cup S_j)'} T[\ell^\infty]$$
        and add it to the set.

    \item[(ii)] If new=false, exit the loop and return the obtained results, forgetting the third element of the triples from $B$, i.e returning just the values $(K,T)$. If new=true, set $k:=k+1$.
   \end{itemize}
\end{enumerate}
\begin{remark}
Note that in the previous algorithm the elements in $S_i$ and $S_j$ are the same only if both coordinates are the same.
\end{remark}

Finally the whole algorithm:

\medskip
{\bf \textsc{Algorithm TorsionGrowth}}

{\sc Input:} An elliptic curve $E/\Q$ and a positive integer $d$.

{\sc Output:} A sequence of all pairs $(K, T)$ of a number field $K$ of degree $d'$ such that $d'|d$ and that $E$ has primitive torsion growth over $K$, together with the group $T:=E(K)_{\tors}$.
\medskip
%
%
%
%

\begin{enumerate}
  \item $R(d,E):=\texttt{Algorithm1(E,d)}$
  \item $A:=\emptyset$\\
  	For $\ell \in R(d,E)$:\\
	 \,$\mbox{\,}\qquad\qquad A:=A\,\cup\,\texttt{Algorithm2(E,$\ell$,d)}$
  \item $B:=\texttt{Algorithm3(A,d)}$
  \item Check whether in $B$ there are pairs $(K_1, T_1)$ and $(K_2,T_2)$ such that $K_1\simeq K_2$ and $T_1\geq T_2$. If yes, remove $(K_2, T_2)$ from $B$.
  \item Return $B$
\end{enumerate}

\section{Computational results}\label{sec:lmfdb}

One of the main motivations for the development of our algorithm is to get computational evidence of how the torsion grows when we consider an elliptic curve defined over $\Q$ base change to a number field of fixed degree.

Our algorithm takes as input an elliptic curve $E$ defined over $\Q$ and a positive integer $d$ and outputs all the pairs $(K,H)$ (up to isomorphism) where $K$ is a number field of degree dividing $d$, $E$ has primitive torsion growth over $K$, and $E(K)_{\tors}\simeq H $. We denote by $\mathcal{H}_{\Q}(d,E)$ the multiset formed by 
the groups $H$ obtained in the above computation. Note that we are allowing the possibility of two (or more) of the torsion subgroups $H$ being isomorphic if the corresponding number fields $K$ are not isomorphic. We call the set $\mathcal{H}_{\Q}(d,E)$ the set of {\it torsion configurations of degree $d$} of the elliptic curve $E/\Q$. We let $\mathcal{H}_{\Q}(d)$ denote the set of $\mathcal{H}_{\Q}(d,E)$ as $E$ runs over all elliptic curves defined over $\Q$ such that $\mathcal{H}_{\Q}(d,E)\ne \{E(\Q)_{\tors}\}$, that is $E$ has torsion growth over a number field of degree $d$. For $S\in \mathcal{H}_{\Q}(d)$ define $N_\Q(S)$ to be the minimum conductor $N_\Q(E)$ such that $\mathcal{H}_{\Q}(d,E)=S$ and we denote by $N_\Q(d)$ the maximum\footnote{Note that the smallest integer $B$ such that for every torsion group $T$ possible over $\Q$ there exists an elliptic curve $E$ with $E(\Q)_{\tors}= T$ and $N_\Q(E)\leq B$ is $B=210$.} of $N_\Q(S)$ for all $S\in\mathcal{H}_\Q(d)$. Note that if we denote the maximum of the cardinality of the sets $S$ when $S\in \mathcal{H}_{\Q}(d)$ by $h_\Q(d)$, then $h_\Q(d)$ gives the maximum number of field extension of degrees dividing $d$ where there is primitive torsion growth. The sets $\mathcal{H}_{\Q}(d)$ have been completely determined for $d=2,3,5,7$ and for any $d$ not divisible by a prime smaller than $11$ (see \cite{GJT15,GJNT15,GJ17,GN?}). From these results, one can read out the value of $h_\Q(d)$ for $d=2,3,5,7$ (see \cite{najtwist} for a different approach to obtain $h_\Q(2)$). For $d=4,6$, exhaustive computations to obtain bounds on the above sets and values have been carried out (see \cite{GJLR18,DGJ18}).

As $d$ grows, all these problems become much more difficult, so it makes sense to obtain lower bounds on some of these sets, where possible. We will obtain such a lower bound for $d\leq 23$, by finding all the possible torsion groups of the $2.483.649$ elliptic curves of conductor less than $400.000$ over number fields of degree up to $23$. We chose to stop at $23$ (although it could probably be feasible to do computations for a few more degrees), as this is the largest degree of number fields that have been included in the LMFDB at the moment of writing of this paper. The algorithm has been implemented in \texttt{Magma} \cite{magma} and can be found in the \href{http://matematicas.uam.es/~enrique.gonzalez.jimenez/research/tables/algorithm/algorithm.html}{\color{blue}online supplement} \cite{MagmaCode}.

Table \ref{table} gives a short overview of our computations. For the sake of simplicity we denote in Table \ref{table} by $(n)$ and $(n,m)$ the groups $\Z/n\Z$ and $\Z/n\Z\times \Z/m\Z$, respectively. The values in the table are:
\begin{itemize}
\item  1\textsuperscript{st} column: degree $d$.
\item 2\textsuperscript{nd} column: the set $\Theta(d)$ consisting of all the possible torsion subgroups $H$ such that there exists an elliptic curve $E/\Q$ and a number field $K$ of degree $d$ such that there is primitive torsion growth over $K$ and such that $E(K)_{\tors}=H$. Or in the other words, the subgroups in $\Phi_\Q(d)$ that do not appear in $\Phi_\Q(d')$ for any proper divisor $d'|d$.
\item 3\textsuperscript{rd} column:  a lower bound of $h_\Q(d)$ (or the exact value, where it is known), the maximum number of field extension of degrees dividing $d$ where there is primitive torsion growth.
\item 4\textsuperscript{th} column:  a lower bound of $N_\Q(d)$, the minimum value such that there exist elliptic curves over $\Q$ of conductor less than $N_\Q(d)$ with every possible torsion configuration over number fields of degree $d$.
\item 5\textsuperscript{th} column:  a lower bound of $\#\mathcal{H}_\Q(d)$, the number of torsion configurations over number fields of degree $d$.
\end{itemize}


\begin{table}[h]
\begin{tabular}{|c|c|c|c|c|} \hline
  $d$ & $\Phi_\Q(d)\setminus\cup_{d'|d, d'<d}{\Phi_\Q(d')}\supseteq$ & $h_\Q(d)$ & ${N}_\Q(d)$ & $\#\mathcal{H}_\Q(d)$ \\ \hline \hline
   $1$ & $\{(1),(2),(3),(4),(5),(6),(7),(8),(9),(10),(12),(2,2),(2,4),(2,6),(2,8)\}$ & $-$ & $-$ & $-$ \\ \hline \hline
   $2$ &   $\{(15),(16),(2,10),(2,12),(3,3),(3,6),(4,4)\}$ &  $4$ & {$3150$} &  $52$ \\ \hline
  $3$ &   $\{(13),(14),(18),(21),(2,14)\}$ &  $3$ &  {$3969$} &  $26$  \\ \hline
$4$ &    $\{(13),(20),(24),(2,16),(4,8),(5,5),(6,6)\}$ &  $\ge 9$ & $\ge 14400$ & $\ge 130$   \\ \hline
  $5$ &   $\{(11),(25)\}$ &  $1$ & {$121$} &  $4$   \\ \hline
 $6$ &   $\{(30),(2,18),(3,9),(3,12),(4,12),(6,6)\}$ & $\ge 9$ & $\ge 10816$ & $\ge 137$  \\ \hline
 $7$ &   $-$ &  $1$ & \ {$26$} &  $1$  \\ \hline
 $8$ &   $\{(17),(21),(30),(32),(2,20),(2,24),(3,12),(4,12)\}$ & $\ge 17$ & $\ge 277440$ & $\ge 275$  \\ \hline
 $9$ &  $\{(19),(26),(27),(28),(36),(42),(2,18)\}$ & $\ge 6$ & $\ge 3969$ & $\ge 34$  \\ \hline
 $10$ & $-$ & $\ge 4$ & $\ge 3150$ & $\ge 58$  \\ \hline
 $12$ & $\left\{\begin{array}{c}( 26 ), ( 28 ), ( 36 ), ( 37 ), ( 42 ), \\( 2, 28 ), ( 2, 30 ), ( 2, 42 ), ( 3, 15 ), ( 3, 21 ),( 5, 10 ), ( 6, 12 )\end{array}\right\}$ & $\ge 19$ & $\ge 18176$ & $\ge 268$  \\ \hline
 $14$ & $-$ & $\ge 4$ &  $\ge 3150$ & $\ge 52$  \\ \hline
 $15$ & $\{( 22 ), ( 50 )\}$ & $\ge 3$ & $\ge 3969$ & $\ge 30$  \\ \hline
 $16$ & $\{( 40 ), ( 48 ), ( 2, 30 ), ( 2, 32 ), ( 3, 15 ), ( 4, 16 ), ( 4, 20 ), ( 5, 15 ), ( 6, 12 ), ( 8, 8 )\}$ & $\ge 25$ & $\ge 277440$ & $\ge 480$  \\ \hline
 $18$ & $\{( 45 ), ( 2, 26 ), ( 2, 36 ), ( 2, 42 ), ( 3, 18 ), ( 3, 21 ), ( 4, 28 ), ( 6, 18 ), ( 7, 7 ), ( 9, 9 )\}$ & $\ge 17$ & $\ge 254016$  & $\ge 192$  \\ \hline
 $20$ & $\{( 22 ), ( 33 ), ( 2, 22 ), ( 5, 10 ), ( 5, 15 )\}$ & $\ge  9 $ &$\ge 14400$  & $\ge 149$  \\ \hline
 $21$ & $\{ ( 43 )\}$ & $\ge 3$ & $\ge 3969$ & $\ge 29 $  \\ \hline
\end{tabular}
\captionof{table}{Bounds on $\Phi_\Q(d)$ for $d\le 23$.}\label{table}
\end{table}
\newpage
\begin{remark}
In Table \ref{table} the degrees over which we know that there is no primitive torsion growth ($d=11,13,17,19,22,23$) have been excluded. The fact that there is no primitive torsion growth over number fields of degree $d=22$ follows from Lemma \ref{lem:22-26}.
\end{remark}

Table \ref{table} gives some useful information to conjecture upon. Note that any group in $\Phi_\Q(d)$ will also arise in $\Phi_\Q(dk)$ for any $k\in \Z^+$ \cite[Theorem 2.1. a)]{bcs}. We conjecture that the groups we found are all that are possible.

\begin{conjecture}\label{main_conjecture} Let $d\leq 23 $ and define $\Theta(d)$ to be the set of groups found in Table \ref{table} for each $d$. Then $\Phi_\Q(d)$ consists of the union of all $\Theta(d')$ such that $d' \mid d$.
\end{conjecture}

\begin{remark}\label{rem}
 The 4\textsuperscript{th} column in Table \ref{table} gives a lower bound for $N_\Q(d)$. When this value is very far from $400.000$ (the bound for the conductor up to which we tested), this might suggest more strongly that the corresponding $\Phi_\Q(d)$ is as stated in Conjecture \ref{main_conjecture}, and one should consider that the case for the conjecture stronger in these cases. This happens for $d\ne 8,16,18$. The 3\textsuperscript{rd} and 5\textsuperscript{th} columns and the values $h_\Q(d)$ and $\# \mathcal H_\Q(d)$ give information about the complexity of the torsion growth and how often it happens over the given degree $d$. The values  seem to grow with the powers of $2$ and $3$ dividing $d$, which is to be expected. The highest values correspond to $d=16,12,18,8$ in that order. In particular, when $d$ is divisible by a power of $2$ these values grow considerably.
\end{remark}

In the \href{http://matematicas.uam.es/~enrique.gonzalez.jimenez/research/tables/algorithm/algorithm.html}{\color{blue}online supplement} \cite{MagmaCode} we give more data about our computations. For each degree $d\le 23$ we include the following:
\begin{itemize}
\item For any $G\in\Phi_\Q(1)$ we include a table with a lower bound for the set $\Phi_\Q(d,G)$.
\item For each torsion configuration $S\in\mathcal{H}_\Q(d)$ obtained, we provide the Cremona label \cite{cremonaweb} of the elliptic curve $E/\Q$ with minimal conductor such that $S=\mathcal{H}_\Q(E,d)$.
\end{itemize}

\begin{remark}
At the moment of writing this paper, each elliptic curve defined over $\Q$ with conductor less than $400.000$ and for any degree $d\le 7$, the data obtained with our algorithm appears in \href{http://www.lmfdb.org/EllipticCurve/Q/}{LMFDB}. We have in plan to include all the data for $d\le 23$. Moreover, all our data is already at the {\it Cremona's Elliptic Curve Data} \cite{cremonaweb} in {\it Table Eleven: Torsion Growth}.
\end{remark}

\subsection{Primitive torsion growth} An interesting question is to restrict our attention to the case of primitive torsion growth of exactly a fixed degree instead of the whole growth over number fields of degree dividing a fixed degree. For a positive integer $d$, we denote by $\Psi_\Q(d)\subseteq \Phi_\Q(d)$ the set of groups, up to isomorphism, that appear as primitive torsion growth of an elliptic curve defined over $\Q$ over a number field of degree $d$. In the same vein, we define $\Psi_\Q(d,G)$, $\mathcal{G}_\Q(d,E)$, $\mathcal{G}_\Q(d)$, $g_\Q(d)$, $M_\Q(d)$ analogously as we did $\Phi_\Q(d,G)$, $\mathcal{H}_\Q(d,E)$, $\mathcal{H}_\Q(d)$, $h_\Q(d)$, $N_\Q(d)$, respectively.

In Table \ref{tablepsi} we include a lower bound for the set $\Psi_\Q(d)$ for $d\le 23$. In particular, in each line the first column is the degree $d$, the second column includes the cyclic groups $\Z/n\Z$, denoted by $(n)$, that we have obtained, and the rest of the columns $\Z/m\Z\times\Z/mn\Z$, denoted by $(m,mn)$, for $2\le m\le 9$.

In Table \ref{psi} we show lower bounds for the values  $g_\Q(d)$, $M_\Q(d)$ and $\#\mathcal{G}_\Q(d)$ for $d\le 23$ non-prime.

Again, in the \href{http://matematicas.uam.es/~enrique.gonzalez.jimenez/research/tables/algorithm/algorithm.html}{\color{blue}online supplement} \cite{MagmaCode} we give more data which gives lower bounds on the sets $\Psi_\Q(d,G)$ and the Cremona labels of the elliptic curves $E/\Q$ with minimal conductor for each torsion configuration in $\mathcal{G}_\Q(d)$ that we have obtained.
\begin{table}[h]
\begin{tabular}{|c|c|c|c|c|c|c|c|c|c|c|}
\hline
$d$ & $(n)$ & $(2,2n)$ & $(3,3n)$ & $(4,4n)$ & $(5,5n)$ & $(6,6n)$ & $(7,7n)$ & $(8,8n)$ & $(9,9n)$\\
\hline
2 &  3-10,12,15,16 &1-6 & 1,2 & 1 & -&- & - & -  &  -\\ \hline
3 &  \begin{tabular}{c}2-4,6,7,9,10,\\ 12-14,18,21\end{tabular} & 1,3,7  &  -&- & -&- & -&- & -\\ \hline
4 &  \begin{tabular}{c}3-6,8,10,12,\\13,15,16,20,24\end{tabular}& 2-6,8 & 1,2 & 1,2 & 1 & 1 & -&- & -\\ \hline
5 &  5,10,11,25 & -&- & -&- & -&- & -& -\\ \hline
6 & \begin{tabular}{c}3,4,6,7,9,10,\\ 12-15,18,21,30 \end{tabular} & 1,3,5-7,9 & 1-4 & 1,3 &-& 1 &-& -& -\\ \hline
7 & 7  & -& -& -& -& -& -& -& -\\ \hline
8 & \begin{tabular}{c}3,5,6,8,10,12,15,16,\\17,20,21,24,30,32\end{tabular} & 2-6,8,10,12 & 1,2,4 & 1-3 & 1 & 1 & -& -& -\\ \hline
9& \begin{tabular}{c}6,7,9,12,14,18,\\19,21,26-28,36,42\end{tabular} & 3,7,9 &- &- &- &- &- &- & -\\ \hline
10 & 5,10,11,15,25 & 5 &- &- &- & -& -& -&-\\ \hline
12 &  \begin{tabular}{c}4,6,7-10,12-15,18,20,\\ 21,24,26,28,30,36,37,42\end{tabular}  & \begin{tabular}{c}2,3,5,6,7,\\9,14,15,21 \end{tabular} & 1-5,7 & 1,3 & 2 & 1,2 &- &- &- \\ \hline
14 & 7  &- &- &- &- &- &- &- &- \\ \hline
15 & 10,22,50 &- &- &- &- & -& -& -&-\\ \hline
16 &\begin{tabular}{c}5,8,10,12,15,16,17\\ 20,21,24,30,32,40,48\end{tabular} & \begin{tabular}{c} 3-6,8,10,\\12,15,16 \end{tabular}& 1,2,4,5 &1-5 & 1,3 & 1,2 & -&1 &- \\ \hline
18 &\begin{tabular}{c} 6,7,9,12,14,18,19,21,\\26-28,30,36,42,45 \end{tabular}& 3,7,9,13,18,21 & 2-4,6,7 & 3,7 & -& 1,3 & 1 & - & 1\\ \hline
20& 5,10,11,15,20,22,25,33 & 5,11 &- & - &1-3& -& -& -&-\\ \hline
21 & 7,14,21,43 & 7&- &- &- &- &- &- &-  \\ \hline
\end{tabular}
\captionof{table}{Bounds on $\Psi_\Q(d)$ for $d\le 23$.}\label{tablepsi}
\end{table}

\begin{table}[h]
\begin{tabular}{|c|c|c|c|c|c|c|c|c|c|c|c|c|c|} \hline
$d$ & $4 $ & $6 $  & $8 $ & 9$ $ & $10 $ & $12 $ & $14 $ & $15 $ & $16 $ & $ 18$ & $20 $ & $21 $\\ \hline
$g_\Q(d)\ge $ & 5 & 5& 9& 3& 1& 6&1 &1 &10 &6 &3 &1 \\ \hline
$M_\Q(d)\ge $   & 18176 & 5184 &223494 & 3969 & 150& 18176& 208& 121& 277440& 254016 & 18176 & 1922\\ \hline
$\#\mathcal{G}_\Q(d)\ge $  &104 &88 &200 &20 &7 &134 &1 &3 &336 &101 &26 &6\\ \hline
\end{tabular}

\captionof{table}{Data for $\Psi_\Q(d)$}\label{psi}
\end{table}
\newpage

Similarly to Conjecture \ref{main_conjecture} we can state the following conjecture in the case of primitive torsion growth:

\begin{conjecture}\label{main_conjecture2} Let $d\leq 23 $ and define $\Omega(d)$ to be the set of groups found in Table \ref{tablepsi} for each $d$. Then $\Psi_\Q(d)=\Omega(d)$.
\end{conjecture}

\begin{remark}
Similarly as with $N_\Q(d)$ in Table \ref{table}, $M_\Q(d)$ can be considered to be a measure of how strongly we should believe $\Psi_\Q(d)=\Omega(d)$ for a particular $d$. The values $g_\Q(d)$ and $\#\mathcal{G}_\Q(d)$ measure how often primitive torsion growth happens and how complex it can be over the given degree $d$. As before, we get more primitive torsion growth and more torsion configurations when $d$ is divisible by $3$, and especially $2$.
\end{remark}

\subsection{Heuristical complexity}

Here we give a heuristical complexity of our algorithm. By the results of \cite{GN?}, we can assume that for a large enough $d$, the largest prime $\ell\in R(d,E)$ will be of size $\sim \sqrt d$.

There are 2 parts in our algorithm that should heuristically have a worst case running time $O(d^{18+\epsilon})$ for a fixed elliptic curve $E$. The first one is checking whether a point of order $\ell$ is divisible by $\ell$ in Algorithm $2$, where $\ell$ is a prime of size $\sqrt d$, in case factorization of the reduction of the primitive $\ell^2$-division polynomial modulo small primes in step $4 (i)$ in Algorithm $2$ always has factors of degree dividing $d$. Then in the worst case, we will need to factor a polynomial of degree approximately $d$ over a number field of degree $d$, which is of complexity $O(d^{18+\epsilon})$ (see \cite{landau}).

The other is checking whether 2 number fields of degree $d$ are isomorphic and similarly checking whether 2 number fields of degree approximately $d$ have compositum of degree dividing $d$. The way we implemented both of these functions (as the built-in MAGMA functions were far too slow) is by factoring the defining polynomial of one field over the other, which again has complexity $O(d^{18+\epsilon})$ (as before, see \cite{landau}).

Each of these operations should be expected to occur $O(d^\epsilon)$ times, which leads us to our expected complexity of $O(d^{18+\epsilon})$.

In practice, for small values of $d$, the only ones for which this problem can be solved in practice, one should expect that algorithm $2$ will be the bottleneck of the computations, as the primes $\ell \in R(d,E)$ can be larger than $\sqrt d$. In the computations we performed, Algorithm 2 took about $75\%$ of the total running time.

\subsection{Timing} We ran our algorithm for all elliptic curves defined over $\Q$ of conductor less than $400.000$ and for degree $d\le 23$ on the Number Theory Warwick Grid, in particular at two computers (\texttt{atkin} and \texttt{lehner}) with 64 CPUs at 2.50 GHz and 128GB of memory RAM each. In Table \ref{timing} we show for each degree $d$ the total time of the whole computation, the maximum time taken for a single elliptic curve, and other statistics. Note that this project used roughly $2.7$ cpu-years of computing time.


\begin{table}[h]
\footnotesize
\begin{tabular}{|c|c|c|c|c|c|c|c|c|c|c|c|c|c|c|c|c|} \hline
 d & 2 & 3 & 4 & 5 & 6 & 7 & 8 & 9 & 10 & 12 & 14 & 15 & 16 & 18 & 20 & 21 \\ \hline
Mode (s) & 0.06 & 0.06 & 0.06 & 0.06 & 0.09 & 0.06 & 0.08 & 0.06 & 0.06 & 0.23 & 0.06 & 0.06 & 0.08 & 0.09 & 0.06 & 0.06 \\ \hline
Median (s) & 0.07 & 0.06 & 0.07 & 0.06 & 0.13 & 0.06 & 0.10 & 0.06 & 0.07 & 4.7& 0.07 & 0.06 & 0.10 & 0.13 & 0.07 & 0.06 \\ \hline
Mean (s)   & 0.08 & 0.06 & 0.15 & 0.06 & 0.17 & 0.06 & 1.1 & 0.13 & 0.1 & 6.5 & 0.08 & 0.07 & 24 & 1.4 & 0.35 & 0.06 \\ \hline
Maximum (s)   & 1.3 & 3.7 & 9.0 & 3.5 & 9.1 & 16 & 98 & 16 & 27 & 110 & 16 & 16 &1200 & 440 & 470 & 17 \\ \hline\hline
Total (h)& 54.4  & 43.5 & 106.4 &  42.2 & 119.6 &  41.6 & 774.7 & 88.2 & 66.8 &  4492.8 & 55.45 &  44.85 &  16339 &  1004 &  241.3 &  43.8 \\ \hline
\end{tabular}
\captionof{table}{Timings for the computations}\label{timing}
\end{table}

\section{On sporadic torsion}
Another motivation for our computations are sporadic points on the modular curves $X_1(m,n)$.

\begin{definition}
Let $m,n$ positive integers such that $m|n$. We say that a degree $d$ non-cuspidal point on the modular curve $X_1(m,n)$ is sporadic if there exists only finitely many degree $d$ points on $X_1(m,n)$.
\end{definition}

Obviously there exists a non-cuspidal sporadic point on $X_1(m,n)$ if and only if $\Z/m\Z\times \Z/n\Z \in \Phi(d) \backslash \Phi^\infty(d)$.

There exist no sporadic points on modular curves $X_1(m,n)$ of degree $d\leq 2$, and hence the aforementioned elliptic curve with $\Z/21\Z$ torsion over a cubic field provides the lowest possible degree of a sporadic point on $X_1(n)$. There are many examples of sporadic points on $X_1(n)$ of degree $\geq 5$, see \cite{VH} for a long list. The fact that many of these in fact correspond to sporadic points follows from \cite[Table 1 and Lemma 1]{DvH}.

It is somewhat surprising that there is no (to our knowledge) known example of a sporadic point on $X_1(m,n)$  for $m \geq 2$. Hence it is interesting to ask what is the lowest possible degree of a sporadic point on $X_1(m,n)$ for $m \geq 2$. During our computation, we find a degree $6$ sporadic non-cuspidal point on  $X_1(4,12)$ about which we will say more in Section \ref{sec:sporadic}.

\subsection{A degree 6 sporadic point on $X_1(4,12)$}

As mentioned in the previous section, during our computations of torsion growth for elliptic curves of conductor less than 400.000, we found two elliptic curve with $\Z/4\Z\times \Z/12\Z$ torsion over a sextic field. By \cite[Theorem 1.1]{DS16}, there are only finitely many such curves over sextic fields, so these curves induce sporadic points on $X_1(4,12)$.

We prove a stronger result below.

\label{sec:sporadic}
\begin{tm}
Let $E$ be an elliptic curve defined over $\Q$ and $K/\Q$ such that $[K:\Q]=6$. If $E(K)_{\tors}=\Z/4\Z\times \Z/12\Z$ then the LMFDB label of $E$ is \href{http://www.lmfdb.org/EllipticCurve/Q/162/d/2}{\texttt{162.d2}} or \href{http://www.lmfdb.org/EllipticCurve/Q/1296/l/2}{\texttt{1296.l2}}. In particular, $j(E)=109503/64$.
\end{tm}

\begin{proof}
Let $E$ be an elliptic curve defined over $\Q$ and $K/\Q$ a sextic field such that $E(K)_{\tors}=\Z/4\Z\times \Z/12\Z$. 
First notice that $E$ does not have CM by \cite[\S 4.6]{clark}. Denote by $G:=E(\Q)_{\tors}$ and $H:=E(K)_{\tors}$. Let $G_2$ (resp. $H_2$) denote the $2$-primary part of $G$ (resp. $H$). Then by the classification of the possible growth of the $2$-primary part of the torsion over sextic fields (cf. \cite[Proposition 6 (b), Table 2]{DGJ18}) we have that $G$ is trivial, $\Z/3\Z$, $\Z/4\Z$, $\Z/12\Z$, or $\Z/2\Z\times \Z/4\Z$. The first two cases occur: if $E$ has LMFDB label \href{http://www.lmfdb.org/EllipticCurve/Q/162/d/2}{\texttt{162.d2}} or  \href{http://www.lmfdb.org/EllipticCurve/Q/1296/l/2}{\texttt{1296.l2}} then $H=\Z/4\Z\times \Z/12\Z$, and $G=\Z/3\Z$ or $G$ is trivial, respectively. Let us remove the other three cases:
\begin{itemize}
\item[$\bullet$] $G\ne \Z/2\Z\times\Z/4\Z$ since if $\Z/2\Z\times\Z/4\Z\subset G$ then $\Z/2\Z\times\Z/12\Z\not\subset H$; see the Remark below \cite[Theorem 7]{GJLR18}.
\item[$\bullet$] $G\ne \Z/12\Z$ since otherwise $G_2=\Z/4\Z$ and $H_2=\Z/4\Z\times\Z/4\Z$. The first author together with Lozano-Robledo, based on the classification of all the possible 2-adic images of Galois representations attached to elliptic curves without CM defined over $\Q$ given by Rouse and Zureick-Brown \cite{rouse}, computed the degree of the field of definition of the $\Z/2^i\Z\times \Z/2^{i+j}\Z$ torsion for $i+j \leq 6$ (cf. \cite[\texttt{2primary\_Ss.txt}]{GJLR17}).  \href{http://matematicas.uam.es/~enrique.gonzalez.jimenez/research/tables/algorithm/sporadic4x12.txt}{\color{blue}Using} the above data it would follow that the number field $K$ would have to have a quadratic subfield and that $E$ would have full $4$-torsion over it. Then $E$ would have $\Z/4\Z\times\Z/12\Z$ torsion over this quadratic field, which is impossible \cite{km,kam}.
\item[$\bullet$] $G\ne \Z/4\Z$. 
     \href{http://matematicas.uam.es/~enrique.gonzalez.jimenez/research/tables/algorithm/sporadic4x12.txt}{\color{blue}Using} the same argument as above, we see that $E$ has full $4$-torsion over a quadratic field.
    Since $\Z/4\Z\times\Z/12\Z\not\in \Phi_\Q(d)$ for $d=2,3,4$, we have that the image of the mod $3$ representation is such that there does not exist a point $P\in E(\overline \Q)[3]$ such that $[\Q(P):\Q]=1$ or $2$. On the other hand, by assumption, there exists a point $R\in E(\overline \Q)[3]$ such that $[\Q(R):\Q]$ divides $6$. Checking for example \cite[Table 1]{GN?}, we see that there is no mod $3$ Galois representation satisfying both these conditions.
\end{itemize}
Now if $G$ is trivial or $G=\Z/3\Z$ we have that $G_2$ is trivial and $H_2=\Z/4\Z\times\Z/4\Z$. We check using \cite{GJLR17} and \cite{rouse} that this happens over a sextic number field if and only if the $2$-adic image correspond to the modular curve \texttt{X20b} (using the notation of \cite{rouse}), implying that there exists a $t\in\Q$ such that $E$ is isomorphic to $E_t$, where:
$$
E_t:y^2=x^3-27 \left(t^2-3\right) \left(t^2-8 t-11\right)^3x+54 \left(t^2-8 t-11\right)^4 \left(t^2-6 t-9\right) \left(t^2+2 t+3\right).
$$
In particular,
$$
j(E_t)=-\frac{4 \left(t^2-3\right)^3 \left(t^2-8 t-11\right)}{(t+1)^4}.
$$
Now we need a point of order $3$ on $E_t$ defined over a subfield of a sextic number field. Checking \cite[Table 1]{GN?} we obtain that this could happen when $G_E(3)$ is \texttt{3Cs.1.1}, \texttt{3B.1.1}, \texttt{3Cs}, \texttt{3B.1.2} or \texttt{3B}. Then, thanks to the classification of mod $3$ Galois representation of \cite[Theorem 1.2]{zyw} we have that $j(E)=J_1(s)$ or $j(E)=J_3(s)$ for some $s\in\Q$, where:
$$
J_1(s)=\frac{27(s+1)^3(s+3)^3(s^2+3)^3}{t^3(t^2+3t+3)^3} \qquad\mbox{and}\qquad J_3(s)=\frac{27(s+1)(s+9)^3}{s^3}.
$$
\begin{itemize}
\item[$\bullet$] $j(E_t)=J_1(s)$. Since $J_1(s)$ is a cube we have to solve the following Diophantine equation over $\Q$:
$$
(t+1)z^3=-4\left(t^2-8 t-11\right).
$$
This equation defines a curve $C$ of genus $2$, which is birational to $C'\,:\,y^2 = x^6-10x^3+27$. The Jacobian of $C'$ has rank $0$ over $\Q$, 
so it is easy to  \href{http://matematicas.uam.es/~enrique.gonzalez.jimenez/research/tables/algorithm/sporadic4x12.txt}{\color{blue}determine} that the points on $C'(\Q)=\{\pm \infty \}$, from which it follows that $C(\Q)=\{\pm \infty\}$. So there do not exist $t,s\in \Q$ satisfying $j(E_t)=J_1(s)$.
\item[$\bullet$] $j(E_t)=J_3(s)$. In this case the equation defines a genus $1$ curve, which is birational 
to the elliptic curve \href{http://www.lmfdb.org/EllipticCurve/Q/48/a/3}
{\texttt{48.a3}} which has Mordell-Weil group over $\Q$ isomorphic to $\Z/2\Z\times\Z/4\Z$. An easy computation shows that the possible $t$ are $7,-5,-1/2$ and $-5/4$. The following table  \href{http://matematicas.uam.es/~enrique.gonzalez.jimenez/research/tables/algorithm/sporadic4x12.txt}{\color{blue}shows} for each $t$ the corresponding elliptic curve (by plugging in $t$ into the equation of $E_t$) and the torsion over $\Q$: \\
$$
\begin{array}{|c|c|c|}
\hline
t & \text{label} & G\\
\hline
7 & \href{http://www.lmfdb.org/EllipticCurve/Q/1296/l/2}{\texttt{1296.l2}} & (1)\\
\hline
-5 &\href{http://www.lmfdb.org/EllipticCurve/Q/1296/l/1}{\texttt{1296.l1}} & (1) \\
\hline
-1/2 & \href{http://www.lmfdb.org/EllipticCurve/Q/162/d/1}{\texttt{162.d1}} & (1) \\
\hline
-5/4 & \href{http://www.lmfdb.org/EllipticCurve/Q/162/d/2}{\texttt{162.d2}} & (3)\\
\hline
\end{array}
$$
Note that for the elliptic curve \href{http://www.lmfdb.org/EllipticCurve/Q/162/d/2}{\texttt{162.d2}} we have already obtained that the torsion over some sextic field is $\Z/4\Z\times\Z/12\Z$. For the remaining curves we check that only \href{http://www.lmfdb.org/EllipticCurve/Q/1296/l/2}{\texttt{1296.l2}} has torsion $\Z/4\Z\times\Z/12\Z$ over a sextic field.

\end{itemize}
\end{proof}
\begin{remark}
  One might try to obtain more sporadic points by running a modification of our algorithm for a large number of elliptic curves $E$ with $j(E)\in \Q$.
\end{remark}

\noindent{\textit{\bf Acknowledgements.} } {We would like to thank Jeremy Rouse and David Zureick--Brown for sharing some useful data. We also thank John Cremona for providing access to computer facilities on the Number Theory Warwick Grid at University of Warwick, where the main part of the computations were done and for doing a massive check of all our computations, in particular rechecking that all the curves have the torsion growth we claim. We are greatly indebted to the referee for a very careful and helpful report that significantly improved all aspects of this paper.
}


\end{document}